\documentclass[11pt]{article}

\usepackage{amsmath,amssymb,amsfonts}
\usepackage{algorithmic}
\usepackage{cite}

\usepackage{algorithm}
\usepackage{subfigure}
\usepackage[usenames]{color}
\usepackage[final,notref,notcite]{showkeys}

        \usepackage[pdftex]{graphicx}
  
\usepackage{url}

\newtheorem{theorem}{Theorem}[section]
\newtheorem{definition}{Definition}[section]
\newtheorem{assumption}{Assumption}[section]
\newtheorem{lemma}{Lemma}[section]

\newtheorem{remark}{Remark}[section]

\newenvironment{proof}[1][Proof]{\textbf{#1.} }{\ \rule{0.5em}{0.5em}}

\newcommand{\tr}{{\mathrm{tr}}}

\newcommand{\bm}[1]{\boldsymbol{#1}}

\newcommand{\mbfa}{\mathbf{A}}

\newcommand{\mbfg}{\mathbf{G}}
\newcommand{\mbfi}{\mathbf{I}}

\newcommand{\mbfu}{\mathbf{U}}
\newcommand{\mbfv}{\mathbf{V}}
\newcommand{\mbfw}{\mathbf{W}}
\newcommand{\mbfx}{\mathbf{X}}
\newcommand{\mbfy}{\mathbf{Y}}
\newcommand{\mbfz}{\mathbf{Z}}

\newcommand{\mbr}{\mathbb{R}}

\newcommand{\mck}{\mathcal{K}}
\newcommand{\mcl}{\mathcal{L}}

\newcommand{\mcs}{\mathcal{S}}

\newcommand{\bmcu}{\bm{\mathcal{U}}}
\newcommand{\bmcv}{\bm{\mathcal{V}}}

\newcommand{\bfx}{{\bf x}}

\newcommand{\bfb}{{\bf b}}

\newcommand{\bfw}{{\bf w}}

\DeclareMathOperator*{\argmin}{argmin}

\usepackage[normalem]{ulem}

\begin{document}

\title{\textbf{Sparse Bilinear Logistic Regression}}
\author{Jianing V. Shi$^{1,2}$\thanks{Corresponding author's email address: jianing@math.ucla.edu}, Yangyang Xu$^{3}$, and Richard G. Baraniuk$^{1}$ \\
\small{$^{1}$ Department of Electrical and Computer Engineering, Rice University}\\
\small{$^{2}$ Department of Mathematics, UCLA}\\
\small{$^{3}$ Department of Computational and Applied Mathematics, Rice University}
}
\date{\today}
\maketitle

\begin{abstract}
In this paper, we introduce the concept of sparse bilinear logistic regression for decision problems involving explanatory variables that are two-dimensional matrices.\ Such problems are common in computer vision, brain-computer interfaces, style/content factorization, and parallel factor analysis.\ The underlying optimization problem is bi-convex; we study its solution and develop an efficient algorithm based on block coordinate descent.\ We provide a theoretical guarantee for global convergence and estimate the asymptotical convergence rate using the Kurdyka-{\L}ojasiewicz inequality.\ A range of experiments with simulated and real data demonstrate that sparse bilinear logistic regression outperforms current techniques in several important applications. 
\end{abstract}

\section{Introduction}

Logistic regression~\cite{hosmer2000} has a long history in decision problems, which are ubiquitous in computer vision~\cite{bishop07}, bioinformatics~\cite{tsuruoka07}, gene classification~\cite{liao07}, and neural signal processing~\cite{parra05}.  Recently sparsity has been introduced into logistic regression to combat the curse of dimensionality in problems where only a subset of explanatory variables are informative~\cite{tibshirani96}.  The indices of the non-zero weights correspond to features that are informative about classification, therefore leading to feature selection.  Sparse logistic regression has many attractive properties, including robustness to noise and logarithmic sample complexity bounds~\cite{ng2004}.

In the classical form of logistic regression, the explanatory variables are treated as i.i.d.\ vectors.  However, in many real-world applications, the explanatory variables take the form of matrices.\ In image recognition tasks~\cite{lecun1998}, for example, each feature is an image.\ Visual recognition tasks for video data often use a feature-based representation, such as the scale-invariant feature transform (SIFT)~\cite{lowe1999} or histogram of oriented gradients (HOG)~\cite{dalal2005}, to construct features for each frame, resulting in histogram-time feature matrices.\ Brain-computer interfaces based on electroencephalography (EEG) make decisions about motor action~\cite{vidal1977} using channel-time matrices.

For these and other applications, bilinear logistic regression~\cite{dyrholm2007} extends logistic regression to explanatory variables that take two-dimensional matrix form.\ The resulting dimensionality reduction of the feature space in turn yields better generalization performance.\ In contrast to standard logistic regression, which collapses each feature matrix into a vector and learns a single weight vector, bilinear logistic regression learns weight factors along each dimension of the matrix to form the decision boundary.\ It has been shown that the unregularized bilinear logistic regression outperforms linear logistic regression in several applications, including brain-computer interfaces~\cite{dyrholm2007}.\ It has also been shown that in certain visual recognition tasks, a support vector machine (SVM) applied in the bilinear feature space outperforms an SVM applied in
the standard linear feature space as well as an SVM applied to a dimensionality-reduced feature space using principle component analysis (PCA)~\cite{pirsiavash2009}.

Bilinear logistic regression has also found application in style and content separation, which can improve the performance of object recognition tasks under various nuisance variables such as orientation, scale, and viewpoint~\cite{tenenbaum2000}.  Bilinear logistic regression identifies subspace projections that factor out informative features and nuisance variables, thus leading to better generalization performance.

Finally, bilinear logistic regression reveals the contributions of different dimensions to classification performance, similarly to parallel factor analysis~\cite{harshman1970}.  This leads to better interpretability of the resulting decision boundary.

In this paper, we introduce sparsity to the bilinear logistic regression model and demonstrate that it improves 
generalization performance in a range of classification problems.\ Our contributions are three-fold.\ First, we propose a sparse bilinear regression model that fuses the key ideas behind both sparse logistic regression and bilinear logistic regression.\ Second, we study the properties of the solution of the bilinear logistic regression problem.\ Third, we develop an efficient algorithm based on block coordinate descent for solving the sparse bilinear regression problem.\ Both the theoretical analysis and the numerical optimization are complicated by the bi-convex nature of the problem, since the solution may become stuck at a non-stationary point.  In contrast to the conventional block coordinate descent method, we solve each subproblem using the proximal method,
which significantly accelerates convergence.\ We also provide a theoretical guarantee for global convergence and estimate the asymptotical convergence rate using a result based on the Kurdyka-{\L}ojasiewicz inequality.

We demonstrate empirically that sparse bilinear logistic regression improves the generalization performance of the classifier under various tasks.  However, due to the non-convexity associated with the bilinear model, it remains a challenge to carry out rigorous statistical analysis using the minimax theory.

\section{Sparse bilinear logistic regression}

\subsection{Problem Definition}

We consider the following problem in this paper: Given $n$ sample-label pairs $\{(\mathbf{X}_i, y_i)\}_{i=1}^n$, where $\mbfx_i \in \mathbb{R}^{s \times t}$ is an explanatory variable in the form of a matrix and $y_i\in\{-1,+1\}$ is a categorical dependent variable, we seek a decision boundary to separate these samples.
  
\subsection{Prior Art}

\subsubsection{Logistic Regression}

The basic form of logistic regression~\cite{hosmer2000} transforms each explanatory variable from a matrix to a vector, $\mathbf{\bar{x}}_i = vec(\mbfx_i) \in \mathbb{R}^{p}$, where $p = st$.  One seeks a hyperplane, defined as $\{\bfx:\bfw^\top\bfx+b=0\}$, to separate these samples. For a new data sample $\mathbf{\bar{x}}_i$, its category can be predicted using a binomial model based on the margin $\bfw^\top\mathbf{\bar{x}}_i+b$. Figure~\ref{fig:lr} illustrates such an idea. 
\begin{figure}[h]
\begin{center}
\includegraphics[width=0.58\columnwidth]{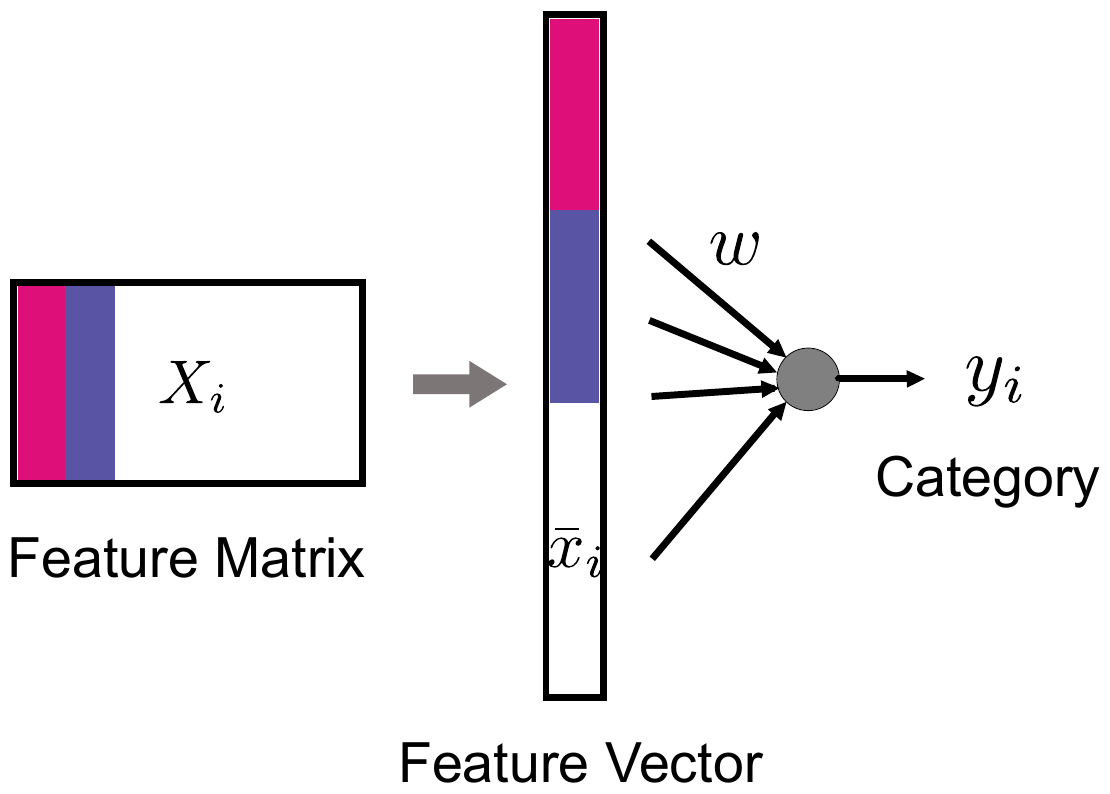}
\end{center}
\vskip -2ex
\caption{Illustration of logistic regression.}
\label{fig:lr}
\end{figure}

Essentially the logistic regression constructs a mapping from the feature vector $\bar{\mathbf{x}}_i$ to the label $y_i$,
\begin{equation*}
\Psi^{LR}: \bfw^{\top} \bar{\mathbf{x}}_i + b \mapsto y_i.
\end{equation*}
Assuming the samples of both classes are i.i.d., the conditional probability for classifier label $y_i$ based on sample $\bar{\mathbf{x}}_i$, according to the logistic model, takes the form of 
\begin{equation*}\label{eq:pr}
p(y_i| \bar{\mathbf{x}}_i,\bfw,b)=\frac{\exp[y_i(\bfw^\top \bar{\mathbf{x}}_i+b)]}{1+\exp[y_i(\bfw^\top \bar{\mathbf{x}}_i+b)]},\quad i = 1,\cdots,n.
\end{equation*}
To perform the maximum likelihood estimation (MLE) of $\bfw$ and $b$, one can minimize the empirical loss function
\begin{equation}\label{eq:ml}
\ell(\bfw,b) = \frac{1}{n}\sum_{i=1}^n \log \left( 1+\exp[-y_i(\bfw^\top \bar{\mathbf{x}}_i+b)] \right).
\end{equation}

\subsubsection{Sparse Logistic Regression}

Sparse logistic regression assumes that only a subset of the decision variables are informative about the classification~\cite{tibshirani96}.  Typically one assumes a sparsity promoting prior on $\bfw$ using the Laplacian prior.  The maximum a posteriori (MAP) estimate for sparse logistic regression can be reduced to an $\ell_1$ minimization problem
\begin{equation}\label{eq:map}
\min_{\bfw,b} \ell(\bfw,b) + \lambda \|\bfw\|_1,
\end{equation}
where $\lambda$ is a regularization parameter.  

In the realm of machine learning, $\ell_1$ regularization exists in various forms of classifiers, including $\ell_1$-regularized logistic regression \cite{tibshirani96}, $\ell_1$-regularized probit regression \cite{figueiredo01,figueiredo03}, $\ell_1$-regularized support vector machines \cite{zhu04}, and $\ell_1$-regularized multinomial logistic regression \cite{krishnapuram05}. 

The $\ell_1$-regularized logistic regression problem is convex but non-differentiable.  There has been very active development on efficient numeric algorithms, including LASSO \cite{tibshirani96}, Gl1ce \cite{lokhorst99}, Grafting \cite{perkins03}, GenLASSO \cite{roth04}, SCGIS \cite{goodman04}, IRLS-LARS \cite{efron04,lee06}, BBR \cite{eyheramendy03,madign05,genkin07}, MOSEK \cite{boyd07}, SMLR \cite{krishnapuram05}, interior-point method \cite{koh07}, FISTA \cite{beck2009}, and HIS \cite{shi10}. 

\subsubsection{Bilinear Logistic Regression}

Bilinear logistic regression was proposed in \cite{dyrholm2007}.  A key insight of bilinear logistic regression is to preserve the matrix structure of the explanatory variables.  The decision boundary is constructed using a weight matrix $\mbfw$, which is further factorized into $\mbfw=\mbfu\mbfv^\top$ with two factors $\mbfu\in\mbr^{s\times r}$ and $\mbfv\in\mbr^{r \times t}$. Figure~\ref{fig:blr} illustrates the concept of bilinear logistic regression.
\vskip 2ex
\begin{figure}[h]
\begin{center}
\includegraphics[width=0.58\columnwidth]{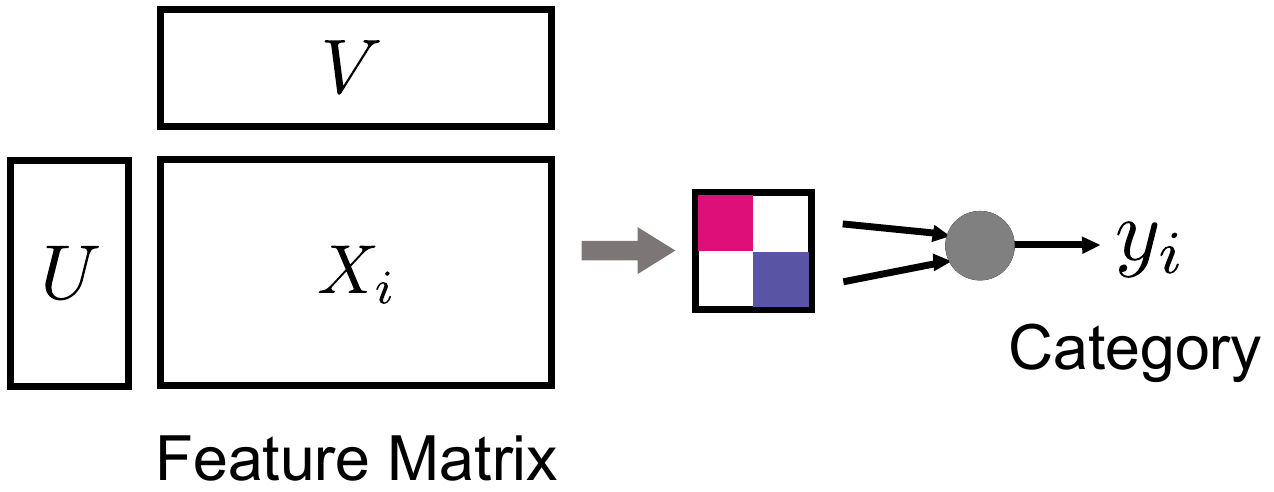}
\end{center}
\vskip -2ex
\caption{Illustration of bilinear logistic regression.}
\label{fig:blr}
\end{figure}
\vskip 1ex
Bilinear logistic regression constructs a new mapping from the feature matrix $\mbfx_i$ to the label $y_i$,
\begin{equation*}
\Psi^{BLR}: \tr(\mbfu^{\top} \mbfx_i \mbfv) + b \mapsto y_i,
\end{equation*}
where $\tr(\mbfa)=\sum_{i}a_{ii}$ for a square matrix $\mbfa$.
Under these settings, the empirical loss function in \eqref{eq:ml} becomes
\begin{equation}\label{eq:bml}
\ell(\mbfu,\mbfv,b)= \frac{1}{n}\sum_{i=1}^n\log \left( 1+\exp[-y_i (\tr(\mbfu^\top\mbfx_i\mbfv)+b)] \right).
\end{equation}

The model \eqref{eq:bml} essentially identifies subspace projections that are maximally informative about classification.  The variational problem generates a low-rank weight matrix $\mbfw \in \mathbb{R}^{s \times t}$ that can be factorized into $\mbfw \in \mathbb{R}^{s \times r}$ and $\mbfv \in \mathbb{R}^{r \times t}$.  One can interpret the mapping from the feature matrix to the label in the following equivalent form,
\begin{equation*}
\Psi^{BLR}: \tr(\mbfw \otimes \mbfx_i) + b \mapsto y_i, \quad \mbfw = \mbfu \mbfv^{\top}, 
\end{equation*}
where $\tr(\mbfw \otimes \mbfx_i) = \sum_{j,k} (\mbfw)_{jk} (\mbfx_i)_{jk}$.

\subsection{Our New Model}

\subsubsection{Sparse Bilinear Logistic Regression}

We introduce sparsity promoting priors on $\mbfu$ and $\mbfv$ and derive the so-called \emph{sparse bilinear logistic regression}.  The corresponding variational problem can be obtained using the MAP estimate,
\begin{equation}\label{eq:main}
\min_{\mbfu,\mbfv,b} \ell(\mbfu,\mbfv,b)+
r_1(\mbfu)+r_2(\mbfv),
\end{equation}
where $r_1$ and $r_2$ are assumed to be convex functions incorporating the priors to promote structures on $\mbfu$ and $\mbfv$, respectively.  Plugging the empirical loss function for bilinear logistic regression, the objective function of sparse bilinear logistic regression becomes
\begin{equation}
\label{eq:sblr}
\min_{\mbfu,\mbfv,b} \frac{1}{n}\sum_{i=1}^n\log \left( 1+\exp[-y_i (\tr(\mbfu^\top\mbfx_i\mbfv)+b)] \right)+r_1(\mbfu)+r_2(\mbfv).
\end{equation}

As for the sparsity promoting priors, in this paper we focus on elastic net regularization~\cite{zhu05} of the form
\begin{subequations}\label{eq:reg}
\begin{align}
r_1(\mbfu)=&\hspace{1mm} \mu_1\|\mbfu\|_1+\frac{\mu_2}{2}\|\mbfu\|_F^2,\label{eq:reg-u}\\
r_2(\mbfv)=&\hspace{1mm} \nu_1\|\mbfv\|_1+\frac{\nu_2}{2}\|\mbfv\|_F^2,\label{eq:reg-v}
\end{align}
\end{subequations}
where $\|\mbfu\|_1\triangleq\sum_{i,j}|u_{ij}|$. Depending on the application, other regularizers can be used. For example, one can use the total variation regularization, which we plan to explore in future work.  

\subsubsection{Why Sparsity?}

The reasons for introducing sparsity promoting priors into bilinear logistic regression are three-fold.

First, according to \cite{dyrholm2007}, one limitation of bilinear logistic regression is the notorious ambiguity in the estimates.  More specifically, the estimated $\hat{\mbfu}$ and $\hat{\mbfv}$ are subject to an arbitrary linear column space transformation
\begin{equation}
\tr(\hat{\mbfu}^{\top} \mbfx_i \hat{\mbfv}) = \tr(\mbfg^{-1}\mbfg \hat{\mbfu}^{\top} \mbfx_i \hat{\mbfv})
= \tr\big( (\hat{\mbfu} \mbfg^{\top})^{\top} \mbfx_i (\hat{\mbfv}{\mbfg}^{-1}) \big),
\end{equation}
where $\mbfg \in \mathbb{R}^{r \times r}$ is an arbitrary full-rank matrix.  Thus the solution to bilinear logistic regression is not unique.  One can overcome such an ambiguity by introducing sparsity promoting priors on weight factors.

Second, bilinear logistic regression was originally motivated by analyzing neuroimaging data~\cite{dyrholm2007}.  The resulting weight factors $\hat{\mbfu}$, $\hat{\mbfv}$ reveal spatial and temporal contributions of neural signal, with respect to certain classification tasks.  Typically the neural sources generating these factors are localized spatially and temporally.  Sparsity leads to feature selection, since the non-zero elements in the weight factors correspond to informative features.  Therefore, it is a reasonable assumption to impose sparsity promoting priors, which can improve the interpretation of the resulting factors.

Third, sparsity improves the generalization performance of the classifier, due to the its robustness to noise and logarithmic sample complexity bounds~\cite{ng2004}.  Even though the statistical analysis based on covering numbers~\cite{ng2004} concerns linear logistic regression models, we envision that such an intuition should generalize to the bilinear model.  We show empirically below that sparsity improves the generalization performance of the classifier in a range of numerical experiments.

\section{Numerical Algorithm to Solve \eqref{eq:main}}

\subsection{Block Coordinate Descent}

We propose an efficient numerical algorithm to solve for the variational problem \eqref{eq:main}.  It is based on the  \emph{block coordinate descent} method, which iteratively updates $(\mbfu,b)$ with $\mbfv$ fixed and then $(\mbfv,b)$ with $\mbfu$ fixed.  The original flavor of block coordinate descent, see \cite{LuoTseng93, Tseng01} and the references therein, alternates between the following two subproblems:
\begin{subequations}\label{eq:bcd}
\begin{align}
(\mbfu^k,\hat{b}^k) &= \argmin_{(\mbfu,b)}\ell(\mbfu,\mbfv^{k-1},b)+r_1(\mbfu), \label{eq:bcd-u}\\
(\mbfv^k,b^k) &= \argmin_{(\mbfv,b)}\ell(\mbfu^k,\mbfv,\hat{b}^k)+r_2(\mbfv) \label{eq:bcd-v}.
\end{align}
\end{subequations} 
The pseudocode for block coordinate descent is summarized in Algorithm~\ref{alg:bcd}.
\begin{algorithm}
\caption{Block Coordinate Descent} \label{alg:bcd}
\begin{algorithmic}
\STATE \textbf{Input:} $\{\mbfx_i,y_i\}_{i=1}^n$
\STATE \textbf{Initialization:} Choose $(\mbfu^{0},\mbfv^{0},b^{0})$
\WHILE{convergence criterion not met}
\STATE Compute $(\mbfu^k,\hat{b}^{k})$ by solving \eqref{eq:bcd-u} 
\STATE Compute $(\mbfv^k,b^{k})$ by solving \eqref{eq:bcd-v} 
\STATE Let $k=k+1$
\ENDWHILE
\end{algorithmic}
\end{algorithm}

Note that even though various optimization methods exist to solve each block, due to the nonlinear form of the empirical loss function $\ell(\cdot)$, solving each block accurately can be computationally expensive.

\subsection{Block Coordinate Proximal Descent}

In order to accelerate computation, we solve each block using the proximal method.  We call the resulting approach the \emph{block coordinate proximal descent} method.  Specifically, at iteration $k$, we perform the following updates:
\begin{subequations}\label{eq:update}
\begin{align}
\mbfu^k=&\argmin_{\mbfu}\langle\nabla_{\mbfu}\ell(\mbfu^{k-1},\mbfv^{k-1},b^{k-1}),\mbfu-\mbfu^{k-1}\rangle
		+\frac{L_u^k}{2}\|\mbfu-\mbfu^{k-1}\|^2_F+r_1(\mbfu),\label{eq:update-u}\\
\hat{b}^{k}=&\argmin_b\langle\nabla_b\ell(\mbfu^{k-1},\mbfv^{k-1},b^{k-1}),b-b^{k-1}\rangle
		+\frac{L_u^k}{2}(b-b^{k-1})^2,\label{eq:update-b1}\\
\mbfv^k=&\argmin_{\mbfv}\langle\nabla_{\mbfv}\ell(\mbfu^{k},\mbfv^{k-1},\hat{b}^{k}),\mbfv-\mbfv^{k-1}\rangle
		+\frac{L_v^k}{2}\|\mbfv-\mbfv^{k-1}\|^2_F+r_2(\mbfv),\label{eq:update-v}\\
b^{k}=&\argmin_b\langle\nabla_b\ell(\mbfu^{k},\mbfv^{k-1},\hat{b}^{k}),b-\hat{b}^{k}\rangle
		+\frac{L_v^k}{2}(b-\hat{b}^{k})^2, \label{eq:update-b2}
\end{align}
\end{subequations}
where $L_u^k$ and $L_v^k$ are stepsize parameters to be specified in Section~\ref{sec:stepsize}. Note that we have decoupled $(\mbfu,b)$-subproblem to \eqref{eq:update-u} and \eqref{eq:update-b1} since the updates of $\mbfu$ and $b$ are independent. Similarly, $(\mbfv,b)$-subproblem has been decoupled to \eqref{eq:update-v} and \eqref{eq:update-b2}.

Denote the objective function of \eqref{eq:main} as $$F(\mbfu,\mbfv,b) \triangleq \ell(\mbfu,\mbfv,b)+r_1(\mbfu)+r_2(\mbfv).$$ Let $F^k \triangleq F(\mbfu^k,\mbfv^k,b^k)$ and $\mbfw^k\triangleq (\mbfu^k,\mbfv^k,b^k)$. We define \emph{convergence criterion} as $q^k \le \epsilon$, where
\begin{equation}\label{eq:relerr}
q^k\triangleq\max \Big\{ \frac{\|\mbfw^k-\mbfw^{k-1}\|_F}{1+\|\mbfw^{k-1}\|_F}, \frac{ | F^k-F^{k-1}| }{1+F^{k-1}} \Big\}
\end{equation}
and $\|\mbfw\|_F^2\triangleq\|\mbfu\|_F^2+\|\mbfv\|_F^2+|b|^2$.

The pseudocode for block coordinate proximal descent is summarized in Algorithm~\ref{alg:bcpd}. 
\begin{algorithm}
\caption{Block Coordinate Proximal Descent} \label{alg:bcpd}
\begin{algorithmic}
\STATE \textbf{Input:} $\{\mbfx_i,y_i\}_{i=1}^n$
\STATE \textbf{Initialization:} Choose $(\mbfu^{0},\mbfv^{0},b^{0})$
\WHILE{convergence criterion not met}
\STATE Compute $(\mbfu^k,\hat{b}^{k})$ by \eqref{eq:update-u} and \eqref{eq:update-b1}
\STATE Compute $(\mbfv^k,b^{k})$ by \eqref{eq:update-v} and \eqref{eq:update-b2}
\STATE Let $k = k+1$
\ENDWHILE
\end{algorithmic}
\end{algorithm}

\subsection{Solving the Subproblems}\label{sec:solve-sub} 

The $b$-subproblems \eqref{eq:update-b1} and \eqref{eq:update-b2} can be simply solved using gradient descent, which can be reduced to
\begin{subequations}\label{eq:sol_b}
\begin{align}
\hat{b}^k  = & \hspace{1mm} {b}^{k-1}-\frac{1}{L_u^k}\nabla_{b}\ell(\mbfu^{k-1},\mbfv^{k-1},{b}^{k-1}),\\ 
{b}^k  = & \hspace{1mm} \hat{b}^{k}-\frac{1}{L_v^k}\nabla_{b}\ell(\mbfu^{k},\mbfv^{k-1},\hat{b}^{k}).
\end{align}
\end{subequations} 

The $\mbfu$-subproblem \eqref{eq:update-u} and $\mbfv$-subproblem \eqref{eq:update-v} are both strongly convex and can be solved by various convex programming solvers.  Since the dimension of input data can be large, it is important to solve the subproblems very efficiently.  The beauty of using the proximal method is its admission for closed-form solutions.  More specifically, for elastic net regularization terms $r_1$ and $r_2$ defined as \eqref{eq:reg}, both \eqref{eq:update-u} and \eqref{eq:update-v} admits closed form solutions
\begin{subequations}\label{eq:sol-uv}
\begin{align}
\mbfu^k=&\mcs_{\tau_u}\left(\frac{L_u^k {\mbfu}^{k-1}-\nabla_{\mbfu}\ell(\mbfu^{k-1},\mbfv^{k-1},{b}^{k-1})}{L_u^k+\mu_2}\right),\\
\mbfv^k=&\mcs_{\tau_v}\left(\frac{L_v^k {\mbfv}^{k-1}-\nabla_{\mbfv}\ell(\mbfu^k,\mbfv^{k-1},\hat{b}^{k})}{L_v^k+\nu_2}\right),
\end{align}
\end{subequations}
where $\tau_u=\frac{\mu_1}{L_u^k+\mu_2}$, $\tau_v=\frac{\nu_1}{L_v^k+\nu_2}$, and $\mcs_{\tau}(\cdot)$ is the component-wise shrinkage defined as
$$\big(\mcs_\tau(\mbfz)\big)_{ij}=\left\{\begin{array}{ll}z_{ij}-\tau,&\text{ if }z_{ij}>\tau;\\
z_{ij}+\tau,&\text{ if }z_{ij}<-\tau;\\
0,&\text{ if }|z_{ij}|\le\tau.\end{array}\right.$$
The proximal method leads to closed-form solution for each subproblem, and the entire algorithm only involves matrix-vector multiplication and component-wise shrinkage operator.  Therefore our numerical algorithm will be computationally efficient. We will corroborate this statement using numerical experiments.



\subsection{Selection of $L_u^k$ and $L_v^k$}\label{sec:stepsize} 

To ensure the sequence generated by Algorithm~\ref{alg:bcpd} attains sufficient decrease in the objective function, $L_u^k$ is typically chosen as a Lipschitz constant of $\nabla_{(\mbfu,b)}\ell(\mbfu,\mbfv^{k-1},b)$ with respect to $(\mbfu,b)$.  More precisely, for all $(\mbfu,b)$ and $(\tilde{\mbfu},\tilde{b})$, we have
\begin{equation}
\|\nabla_{(\mbfu,b)}\ell(\mbfu,\mbfv^{k-1},b)-\nabla_{(\mbfu,b)}\ell(\tilde{\mbfu},\mbfv^{k-1},\tilde{b})\|_F
\le L_u^k\|(\mbfu,b)-(\tilde{\mbfu},\tilde{b})\|_F,
\end{equation}
where $\|(\mbfu,b)\|_F:=\sqrt{\|\mbfu\|_F^2+b^2}$. Similarly, $L_v^k$ can be chosen as a Lipschitz constant of $\nabla_{(\mbfv,b)}\ell(\mbfu^k,\mbfv,b)$ with respect to $(\mbfv,b)$. The next lemma shows that the two partial gradients $\nabla_{(\mbfu,b)}\ell(\mbfu,\mbfv,b)$ and $\nabla_{(\mbfv,b)}\ell(\mbfu,\mbfv,b)$ are Lipschitz continuous with constants dependent on $\mbfu$ and $\mbfv$ respectively. 

\begin{lemma}\label{lem:lip}
The partial gradients $\nabla_{(\mbfu,b)}\ell(\mbfu,\mbfv,b)$ and $\nabla_{(\mbfv,b)}\ell(\mbfu,\mbfv,b)$ are Lipschitz continuous with constants
\begin{subequations}\label{eq:lip}
\begin{align}
L_{u}&=\frac{\sqrt{2}}{n}\sum_{i=1}^n\big(\|\mbfx_i\mbfv\|_F+1\big)^2,\label{eq:lip-u}\\
L_{v}&=\frac{\sqrt{2}}{n}\sum_{i=1}^n\big(\|\mbfx_i^\top\mbfu\|_F+1\big)^2,\label{eq:lip-v}
\end{align}
\end{subequations}
\end{lemma}
\begin{proof}
By straightforward calculation, we have
\begin{subequations}
\begin{align}
\nabla_{\mbfu}\ell(\mbfu,\mbfv,b)&=-\frac{1}{n}\sum_{i=1}^n\left(1+\exp\big[y_i\big(\tr(\mbfu^\top\mbfx_i\mbfv)+b\big)\big]\right)^{-1}y_i\mbfx_i\mbfv,\label{eq:grad_u}\\
\nabla_{\mbfv}\ell(\mbfu,\mbfv,b)&=-\frac{1}{n}\sum_{i=1}^n\left(1+\exp\big[y_i\big(\tr(\mbfu^\top\mbfx_i\mbfv)+b\big)\big]\right)^{-1}y_i\mbfx_i^\top\mbfu,\label{eq:grad_v}\\
\nabla_b\ell(\mbfu,\mbfv,b)&=-\frac{1}{n}\sum_{i=1}^n\left(1+\exp\big[y_i\big(\tr(\mbfu^\top\mbfx_i\mbfv)+b\big)\big]\right)^{-1}y_i.\label{eq:grad_b}
\end{align}
\end{subequations}
For any $(\mbfu,b)$ and $(\tilde{\mbfu},\tilde{b})$, we have
\begin{eqnarray*}
\lefteqn{\|\nabla_{(\mbfu,b)}\ell(\mbfu,\mbfv,b)-\nabla_{(\mbfu,b)}\ell(\tilde{\mbfu},\mbfv,\tilde{b})\|_F}\\
&\le&\frac{1}{n}\sum_{i=1}^n\left|\left(1+\exp\big[y_i\big(\tr(\mbfu^\top\mbfx_i\mbfv)+b\big)\big]\right)^{-1}
-\left(1+\exp\big[y_i\big(\tr(\tilde{\mbfu}^\top\mbfx_i\mbfv)+\tilde{b}\big)\big]\right)^{-1}\right|\\
&&\big(\|\mbfx_i\mbfv\|_F+1\big)\\
&\le&\frac{1}{n}\sum_{i=1}^n\left(\|\mbfu-\tilde{\mbfu}\|_F\|\mbfx_i\mbfv\|_F+|b-\tilde{b}|\right)\big(\|\mbfx_i\mbfv\|_F+1\big)\\
&\le&\frac{1}{n}\sum_{i=1}^n\big(\|\mbfx_i\mbfv\|_F+1\big)^2\left(\|\mbfu-\tilde{\mbfu}\|_F+|b-\tilde{b}|\right)\\
&\le&\frac{\sqrt{2}}{n}\sum_{i=1}^n\big(\|\mbfx_i\mbfv\|_F+1\big)^2\|(\mbfu,b)-(\tilde{\mbfu},\tilde{b})\|_F,
\end{eqnarray*}
where in the third inequality we have used the inequality 
$$|(1+e^s)^{-1}-(1+e^q)^{-1}|\le |s-q|$$
and the last inequality follows from 
$$\|\mbfu-\tilde{\mbfu}\|_F+|b-\tilde{b}|\le\sqrt{2}\|(\mbfu,b)-(\tilde{\mbfu},\tilde{b})\|_F$$
by the Cauchy-Schwarz inequality. This completes the proof of \eqref{eq:lip-u}, and \eqref{eq:lip-v} can be shown in the same way.
\end{proof}

However, $L_u^k$ and $L_v^k$ chosen in such a manner may be too large, slowing convergence.  Therefore we have chosen to use an alternative and efficient way to dynamically update them.  Specifically, we let 
\begin{equation}\label{dynLu}
L_u^k=\max(L_{\min}, L_u^{k-1}\eta^{n_u^k})
\end{equation}
where $L_{\min}>0$, $\eta>1$, and $n_u^k\ge -1$ is the smallest integer such that
\begin{eqnarray}\label{chooseLu}
\lefteqn{\ell(\mbfu^k,\mbfv^{k-1},\hat{b}^k)}\nonumber\\
&\le& \ell(\mbfu^{k-1},\mbfv^{k-1},b^{k-1})\nonumber\\
&&+\langle\nabla_{\mbfu}\ell(\mbfu^{k-1},\mbfv^{k-1},b^{k-1}),\mbfu^k-\mbfu^{k-1}\rangle 
  +\langle\nabla_b\ell(\mbfu^{k-1},\mbfv^{k-1},b^{k-1}),\hat{b}^k-b^{k-1}\rangle\nonumber\\
&&+\frac{L_u^k}{2}\|\mbfu^k-\mbfu^{k-1}\|^2_F+\frac{L_u^k}{2}(\hat{b}^k-b^{k-1})^2,
\end{eqnarray}
and let
\begin{equation}\label{dynLv}
L_v^k=\max(L_{\min},L_v^{k-1}\eta^{n_v^k}),
\end{equation}
where $n_v^k\ge -1$ is the smallest integer such that
\begin{eqnarray}\label{chooseLv}
\lefteqn{\ell(\mbfu^k,\mbfv^{k},b^k)}\nonumber\\
&\le&\hspace{1mm} \ell(\mbfu^{k},\mbfv^{k-1},\hat{b}^{k})\nonumber\\
&&+\langle\nabla_{\mbfv}\ell(\mbfu^{k},\mbfv^{k-1},\hat{b}^{k}),\mbfv^k-\mbfv^{k-1}\rangle
    +\langle\nabla_b\ell(\mbfu^{k},\mbfv^{k-1},\hat{b}^{k}),b^{k}-\hat{b}^k\rangle\nonumber\\
&&+ \frac{L_v^k}{2}\|\mbfv^k-\mbfv^{k-1}\|^2_F+\frac{L_v^k}{2}(b^{k}-\hat{b}^k)^2.
\end{eqnarray}


The inequalities \eqref{chooseLu} and \eqref{chooseLv} guarantee sufficient decrease of the objective and are required for convergence. If $L_u^k$ and $L_v^k$ are taken as Lipschitz constants of $\nabla_{(\mbfu,b)}\ell(\mbfu,\mbfv^{k-1},b)$ and $\nabla_{(\mbfv,b)}\ell(\mbfu^k,\mbfv,b)$, then the two inequalities must hold. In our dynamical updating rule, note that in \eqref{dynLu} and \eqref{dynLv}, we allow $n_u^k$ and $n_v^k$ to be negative, namely, $L_u^k$ and $L_v^k$ can be smaller than their previous values.  Moreover, $n_u^k$ and $n_v^k$ must be finite if the sequence $\{(\mbfu^k,\mbfv^k)\}$ is bounded, and thus the updates in \eqref{dynLu} and \eqref{dynLv} are well-defined.

\section{Convergence Analysis}

We now establish the global convergence of the block coordinate proximal descent algorithm for sparse bilinear logistic regression, as well as estimate its asymptotic convergence rate.  

Our analysis mainly follows \cite{xu2013}, which establishes global convergence of the \emph{cyclic} block coordinate proximal method assuming the Kurdyka-{\L}ojasiewicz inequality (see Definition \ref{def:kl} below). Since our algorithm updates $b$-block twice during each iteration, its convergence result cannot be obtained directly from \cite{xu2013}. The work \cite{TsengYun2009} also establishes global convergence results with rate estimation for the block coordinate proximal method. However, it assumes the so-called \emph{local Lipschitzian error bound}, which is not known to hold for our problem. Throughout our analysis, we make the following assumption.

\begin{assumption}\label{assump1}
We assume the objective function $F$ is lower bounded and the problem \eqref{eq:main} has at least one stationary point.  In addition, we assume the sequence $\{\mbfw^k\}$ is bounded.
\end{assumption}

\begin{remark}
According to \eqref{eq:lip}, $L_u^k,L_v^k$ must be bounded if $\{\mbfw^k\}$ is bounded. In addition, for the regularization terms, $r_1$ set by \eqref{eq:reg-u} and $r_2$ taken as \eqref{eq:reg-v}, then $F$ is lower bounded by \emph{zero}, and \eqref{eq:main} has at least one solution.
\end{remark}

\begin{theorem}[Subsequence Convergence]\label{thm:subconvg}
Under Assumption \ref{assump1},
let $\{\mbfw^k\}$ be the sequence generated from Algorithm \ref{alg:bcpd}. Then any limit point $\bar{\mbfw}$ of $\{\mbfw^k\}$ is a stationary point of \eqref{eq:main}.
\end{theorem}
\begin{proof}
From Lemma 2.3 of~\cite{beck2009}, we have
$$F(\mbfw^{k-1})-F(\mbfu^k,\hat{b}^k,\mbfv^{k-1})\ge\frac{L_u^k}{2}\big(\|\mbfu^{k-1}-\mbfu^k\|_F^2+|b^{k-1}-\hat{b}^k|^2\big),$$
and
$$F(\mbfu^k,\hat{b}^k,\mbfv^{k-1})-F(\mbfw^{k})\ge\frac{L_v^k}{2}\big(\|\mbfv^{k-1}-\mbfv^k\|_F^2+|\hat{b}^k-b^k|^2\big).$$
Assume $\min(L_u^k,L_v^k)\ge L_{\min}$ for all $k$.
Summing up the above two inequality gives
\begin{equation}\label{eq:app-diff}
F(\mbfw^{k-1})-F(\mbfw^k)\ge \frac{L_{\min}}{2}\big(\|\mbfu^{k-1}-\mbfu^k\|_F^2+\|\mbfv^{k-1}-\mbfv^k\|_F^2+|b^{k-1}-\hat{b}^k|^2+|\hat{b}^k-b^k|^2\big),
\end{equation}
which yields
$$F(\mbfw^0)-F(\mbfw^N)\ge\sum_{k=1}^N\big(\|\mbfu^{k-1}-\mbfu^k\|_F^2+\|\mbfv^{k-1}-\mbfv^k\|_F^2+|b^{k-1}-\hat{b}^k|^2+|\hat{b}^k-b^k|^2\big).$$
Letting $N\to\infty$ and observing $F\ge0$, we have
$$\sum_{k=1}^\infty\big(\|\mbfu^{k-1}-\mbfu^k\|_F^2+\|\mbfv^{k-1}-\mbfv^k\|_F^2+|b^{k-1}-\hat{b}^k|^2+|\hat{b}^k-b^k|^2\big)\le \infty.$$
Hence, $\mbfw^k-\mbfw^{k-1}\to\mathbf{0}$.

Let $\bar{\mbfw}$ be a limit point. Hence, there exists a subsequence $\{\mbfw^k\}_{k\in\mck}$ converging to $\bar{\mbfw}$. Passing to another subsequence, we can assume that $\{L_u^k\}_{k\in\mck}$ and $\{L_v^k\}_{k\in\mck}$ converge to $\bar{L}_u$ and $\bar{L}_v$ respectively. Note that $\{\mbfw^{k-1}\}_{k\in\mck}$ also converges to $\bar{\mbfw}$ and $\{\hat{b}^k\}_{k\in\mck}\to\bar{b}$. Letting $k\in\mck$ and $k\to\infty$ in  \eqref{eq:update-u}, we have
$$\bar{\mbfu}=\argmin_{\mbfu}\langle\nabla_{\mbfu}\ell(\bar{\mbfu},\bar{\mbfv},\bar{b}),\mbfu-\bar{\mbfu}\rangle+\frac{\bar{L}_u}{2}\|\mbfu-\bar{\mbfu}\|^2_F+r_1(\mbfu),$$
which implies
$\mathbf{0}\in\nabla_{\mbfu}\ell(\bar{\mbfu},\bar{\mbfv},\bar{b})+\partial r_1(\bar{\mbfu})$. Similarly, one can show $\mathbf{0}\in\nabla_{\mbfv}\ell(\bar{\mbfu},\bar{\mbfv},\bar{b})\\+\partial r_2(\bar{\mbfv})$ and $\nabla_{b}\ell(\bar{\mbfu},\bar{\mbfv},\bar{b})=0$. Hence, $\bar{\mbfw}$ is a critical point.
\end{proof}

In order to establish global convergence, we utilize Kurdyka-{\L}ojasiewicz inequality defined below~\cite{lojasiewicz1993geometrie, kurdyka1998gradients, bolte2007lojasiewicz}.  

\begin{definition}[Kurdyka-{\L}ojasiewicz Inequality]\label{def:kl}
A function $F$ is said to satisfy the Kurdyka-{\L}ojasiewicz inequality at point $\bar{\mbfw}$, if there exists $\theta\in[0,1)$ such that
\begin{equation}\label{eq:KL}\frac{|F(\mbfw)-F(\bar{\mbfw})|^\theta}{\text{dist}(\mathbf{0},\partial F(\mbfw))}\end{equation}
is bounded for any $\mbfw$ near $\bar{\mbfw}$, where $\partial F(\mbfw)$ is the limiting subdifferential~\cite{rockafellar1998variational} of $F$ at $\mbfw$, and $\text{dist}(\mathbf{0},\partial F(\mbfw)) \triangleq \min\{\|\mbfy\|_F: \mbfy\in\partial F(\mbfw)\}$.
\end{definition}

\begin{theorem}[Global Convergence]\label{thm:glbconvg}
\label{thm:convg}
Suppose Assumption \ref{assump1} holds and $F$ satisfies the Kurdyka-{\L}ojasiewicz inequality at a limit point $\bar{\mbfw}$ of $\{\mbfw^k\}$, then $\mbfw^k$ converges to $\bar{\mbfw}$. 
\end{theorem}
\begin{proof}
The boundedness of $\{\mbfw^k\}$ implies that all intermediate points are bounded. Hence, there exists a constant $L_{\max}$ such that $L_u^k, L_v^k\le L_{\max}$ for all $k$, and also there is a constant $L_G$ such that for all $k$
\begin{subequations}\label{eq:app-lip}
\begin{align}
\|\nabla_\mbfu\ell(\mbfw^k)-\nabla_\mbfu\ell(\mbfw^{k-1})\|_F\le&L_G\|\mbfw^k-\mbfw^{k-1}\|_F,\\
\|\nabla_\mbfv\ell(\mbfw^k)-\nabla_\mbfv\ell(\mbfu^k,\mbfv^{k-1},\hat{b}^k)\|_F\le & L_G\|\mbfw^k-(\mbfu^k,\mbfv^{k-1},\hat{b}^k)\|_F,\\
\|\nabla_b\ell(\mbfw^k)-\nabla_b\ell(\mbfu^k,\mbfv^{k-1},\hat{b}^k)\|_F\le & L_G\|\mbfw^k-(\mbfu^k,\mbfv^{k-1},\hat{b}^k)\|_F.
\end{align}
\end{subequations}

Let $\bar{\mbfw}$ be a limit point of $\{\mbfw^k\}$ and assume $F$ satisfies KL-inequality within $\mathbb{B}_\rho(\bar{\mbfw})\triangleq\{\mbfw: \|\mbfw-\bar{\mbfw}\|_F\le\rho\}$, namely, there exists constants $0\le\theta<1$ and $C>0$ such that
\begin{equation}\label{eq:F-KL}
\frac{|F(\mbfw)-F(\bar{\mbfw})|^\theta}{\text{dist}(\mathbf{0},\partial F(\mbfw))}\le C, \quad \forall \mbfw\in\mathbb{B}_\rho(\bar{\mbfw}).
\end{equation}  Noting $\mbfw^k-\mbfw^{k-1}\to\mathbf{0}$, $|b^k-\hat{b}^k|\to 0$, and the continuity of $\phi(s)=s^{1-\theta}$, we can take sufficiently large $k_0$ such that  
\begin{equation}\label{eq:app-cond}
2\|\mbfw^{k_0}-\mbfw^{k_0+1}\|_F+\|\bar{\mbfw}-\mbfw^{k_0}\|_F+|b^{k_0+1}-\hat{b}^{k_0+1}|+\frac{1}{\tilde{C}^2}\phi(F(\mbfw^{k_0})-F(\bar{\mbfw}))\le\rho,
\end{equation}
where $\tilde{C}=\sqrt{\frac{(1-\theta)L_{\min}}{8C\cdot (3L_G+2L_{\max})}}$.
Without loss of generality, we assume $k_0=0$ (i.e., take $\mbfw^{k_0}$ as starting point), since the convergence of $\{\mbfw^k\}_{k\ge 0}$ is equivalent to that of $\{\mbfw^k\}_{k\ge k_0}$. In addition, we denote $F_k=F(\mbfw^k)-F(\bar{\mbfw})$ and note $F_k\ge 0$ from the non-increasing monotonicity of $\{F(\mbfw^k)\}$.

From \eqref{eq:update}, we have
\begin{subequations}\label{eq:app-pg}
\begin{align}
-\nabla_{\mbfu}\ell(\mbfw^{k-1})+\nabla_\mbfu\ell(\mbfw^k)-L_u^k(\mbfu^k-\mbfu^{k-1})&\in\partial r_1(\mbfu^k)+\nabla_\mbfu\ell(\mbfw^k),\\
-\nabla_{\mbfv}\ell(\mbfu^k,\mbfv^{k-1},\hat{b}^k)+\nabla_\mbfv\ell(\mbfw^k)-L_v^k(\mbfv^k-\mbfv^{k-1})&\in\partial r_2(\mbfv^k)+\nabla_\mbfv\ell(\mbfw^k),\\
-\nabla_b\ell(\mbfu^k,\mbfv^{k-1},\hat{b}^k)+\nabla_b\ell(\mbfw^k)-L_v^k(b^k-\hat{b}^k)&=\nabla_b\ell(\mbfw^k).
\end{align}
\end{subequations}
Hence,
\begin{eqnarray}\label{eq:KL2}
\lefteqn{\text{dist}(\mathbf{0},\partial F(\mbfw^k))}\nonumber\\
&\le&\|\nabla_\mbfu\ell(\mbfw^k)-\nabla_\mbfu\ell(\mbfw^{k-1})\|_F+L_u^k\|\mbfu^k-\mbfu^{k-1}\|_F+\|\nabla_\mbfv\ell(\mbfw^k)-\nabla_\mbfv\ell(\mbfu^k,\mbfv^{k-1},\hat{b}^k)\|_F\nonumber\\
&& +L_v^k\|\mbfv^k-\mbfv^{k-1}\|_F+\|\nabla_b\ell(\mbfw^k)-\nabla_b\ell(\mbfu^k,\mbfv^{k-1},\hat{b}^k)\|_F+L_v^k|b^k-\hat{b}^k|\nonumber\\
&\le& (3L_G+2L_{\max})\big(\|\mbfw^k-\mbfw^{k-1}\|_F+|b^k-\hat{b}^k|\big).
\end{eqnarray}
Note that \eqref{eq:app-diff} implies
$$F_k-F_{k+1}\ge \frac{L_{\min}}{4}\big(\|\mbfw^{k+1}-\mbfw^k\|_F^2+|b^{k+1}-\hat{b}^{k+1}|^2\big).$$

Assume $\mbfw^k\in\mathbb{B}_\rho(\bar{\mbfw})$ for $0\le k\le N$. We go to show $\mbfw^{N+1}\in\mathbb{B}_\rho(\bar{\mbfw})$. By the concavity of $\phi(s)=s^{1-\theta}$ and KL-inequality \eqref{eq:F-KL}, we have
\begin{equation}\label{eq:app-kl}\phi(F_k)-\phi(F_{k+1})\ge\phi'(F_k)(F_k-F_{k+1})\ge\frac{(1-\theta)L_{\min}\big(\|\mbfw^{k+1}-\mbfw^k\|_F^2+|b^{k+1}-\hat{b}^{k+1}|^2\big)}{4C\cdot (3L_G+2L_{\max})\big(\|\mbfw^k-\mbfw^{k-1}\|_F+|b^k-\hat{b}^k|\big)},
\end{equation}
which together with Cauchy-Schwarz inequality gives
\begin{equation}\label{eq:app-key}\tilde{C}\big(\|\mbfw^k-\mbfw^{k+1}\|_F+|b^{k+1}-\hat{b}^{k+1}|\big)\le\frac{\tilde{C}}{2}\big(\|\mbfw^{k-1}-\mbfw^k\|_F+|b^k-\hat{b}^k|\big)+\frac{1}{2\tilde{C}}\big(\phi(F_k)-\phi(F_{k+1})\big).
\end{equation}
Summing up the above inequality gives
\begin{equation}\label{eq:app-sum}\frac{\tilde{C}}{2}\sum_{k=1}^N\big(\|\mbfw^k-\mbfw^{k+1}\|_F+|b^{k+1}-\hat{b}^{k+1}|\big)\le \frac{\tilde{C}}{2}\big(\|\mbfw^0-\mbfw^{1}\|_F+|b^1-\hat{b}^1|\big)+\frac{1}{2\tilde{C}}\big(\phi(F_0)-\phi(F_{N+1})\big).
\end{equation}
Hence,
\begin{eqnarray}
\lefteqn{\|\mbfw^{N+1}-\bar{\mbfw}\|_F}\nonumber\\
&\le& \sum_{k=1}^N\|\mbfw^k-\mbfw^{k+1}\|_F+\|\mbfw^0-\mbfw^1\|_F+\|\bar{\mbfw}-\mbfw^0\|_F\nonumber\\
&\le& 2\|\mbfw^0-\mbfw^1\|_F+\|\bar{\mbfw}-\mbfw^0\|_F+|b^1-\hat{b}^1|+\frac{1}{\tilde{C}^2}\phi(F_0)
\le \rho,
\end{eqnarray}
where the last inequality is from \eqref{eq:app-cond}. Hence, $\mbfw^{N+1}\in\mathbb{B}_\rho(\bar{\mbfw})$, and by induction, $\mbfw^{k}\in\mathbb{B}_\rho(\bar{\mbfw})$ for all $k$.
Therefore, \eqref{eq:app-sum} holds for all $N$. Letting $N\to\infty$ in \eqref{eq:app-sum} yields
$$\sum_{k=1}^\infty\|\mbfw^k-\mbfw^{k+1}\|_F<\infty.$$
Therefore $\{\mbfw^k\}$ is a Cauchy sequence and thus converges to the limit point $\bar{\mbfw}$.
\end{proof}

\begin{remark}
Note that the logistic function $\ell$ is real analytic. If $r_1$ and $r_2$ are taken as in \eqref{eq:reg}, then they are semi-algebraic functions~\cite{bochnak1998}, and, according to~\cite{xu2013}, $F$ satisfies the Kurdyka-{\L}ojasiewicz inequality at every point.  
\end{remark}

\begin{theorem}[Convergence Rate]\label{thm:rate}
Depending on $\theta$ in \eqref{eq:KL}, we have the following convergence rates:
\begin{enumerate}
\item If $\theta=0$, then $\mbfw^k$ converges to $\bar{\mbfw}$ in finite iterations;
\item If $\theta\in(0,\frac{1}{2}]$, then $\mbfw^k$ converges to $\bar{\mbfw}$ at least linearly, i.e., $\|\mbfw^k-\bar{\mbfw}\|_F\le C\tau^k$ for some positive constants $C$ and $\tau<1$;
\item If $\theta\in(\frac{1}{2},1)$, then $\mbfw^k$ converges to $\bar{\mbfw}$ at least sublinearly. Specifically, $\|\mbfw^k-\bar{\mbfw}\|_F\le Ck^{-\frac{1-\theta}{2\theta-1}}$ for some constant $C>0$.
\end{enumerate}
\end{theorem}
\begin{proof}
We estimate the convergence rates for different $\theta$ in \eqref{eq:F-KL}.

\textbf{Case 1: $\theta=0$}. We claim $\mbfw^k$ converges to $\bar{\mbfw}$ in finite iterations, i.e., there is $k_0$ such that $\mbfw^k=\bar{\mbfw}$ for all $k\ge k_0$. Otherwise, $F(\mbfw^k)>F(\bar{\mbfw})$ for all $k$ since if $F(\mbfw^{k_0})=F(\bar{\mbfw})$ then $\mbfw^k=\bar{\mbfw}$ for all $k\ge k_0$. By KL-inequality \eqref{eq:F-KL}, we have $C\cdot\text{dist}(\mathbf{0},\partial F(\mbfw^k))\ge 1$ for all $k$. However, \eqref{eq:app-pg} indicates $\text{dist}(\mathbf{0},\partial F(\mbfw^k))\to 0$ as $k\to\infty$. Therefore, if $\theta=0$, then $\mbfw^k$ converges to $\bar{\mbfw}$ in finite iterations.

\textbf{Case 2: $\theta\in(0,\frac{1}{2}]$.} Denote $S_N=\sum_{k=N}^\infty\big(\|\mbfw^k-\mbfw^{k+1}\|_F+|b^{k+1}-\hat{b}^{k+1}|\big)$. Note that \eqref{eq:app-key} holds for all $k$. Summing \eqref{eq:app-key} over $k$ gives
$S_N\le S_{N-1}-S_N+\frac{1}{2\tilde{C}^2}F_N^{1-\theta}.$ By \eqref{eq:F-KL} and \eqref{eq:KL2}, we have
$$F_N^{1-\theta}=(F_N^\theta)^{\frac{1-\theta}{\theta}}\le \big(C\cdot(3L_G+2L_{\max})\big)^{\frac{1-\theta}{\theta}}(S_{N-1}-S_N)^{\frac{1-\theta}{\theta}}.$$
Hence,
\begin{equation}\label{eq:app-key2}
S_N\le S_{N-1}-S_N+\hat{C}(S_{N-1}-S_N)^{\frac{1-\theta}{\theta}},
\end{equation}
where $\hat{C}=\frac{1}{2\tilde{C}^2}\big(C\cdot(3L_G+2L_{\max})\big)^{\frac{1-\theta}{\theta}}$. Note that $S_{N-1}-S_N\le 1$ as $N$ is sufficiently large, and also $\frac{1-\theta}{\theta}\ge 1$ when $\theta\in(0,\frac{1}{2}]$. Therefore, $(S_{N-1}-S_N)^{\frac{1-\theta}{\theta}}\le S_{N-1}-S_N$, and thus \eqref{eq:app-key2} implies $S_N\le (1+\hat{C})(S_{N-1}-S_N)$. Hence, $S_N\le\frac{1+\hat{C}}{2+\hat{C}}S_{N-1}\le\big(\frac{1+\hat{C}}{2+\hat{C}}\big)^N S_0$. Noting that $\|\mbfw^N-\bar{\mbfw}\|_F\le S_N$, we have
$$\|\mbfw^N-\bar{\mbfw}\|_F\le\big(\frac{1+\hat{C}}{2+\hat{C}}\big)^N S_0.$$

\textbf{Case 3: $\theta\in(\frac{1}{2},1)$.} Note $\frac{1-\theta}{\theta}<1$. Hence, $\eqref{eq:app-key2}$ implies that
$$S_N\le (1+\hat{C})(S_{N-1}-S_N)^{\frac{1-\theta}{\theta}}.$$ 
Through the same argument in the proof of Theorem 2 of~\cite{attouch2009}, we can show
$$S_N\le c\cdot N^{-\frac{1-\theta}{2\theta-1}},$$
for some constant $c$. This completes the proof.
\end{proof}

\begin{remark}
Note that the value of $\theta$ depends not only on $F$ but also on $\bar{\mbfw}$. The paper~\cite{xu2013} gives estimates for different classes of functions. Since the limit point is not known ahead, we cannot estimate $\theta$. However, our numerical results in Section \ref{sec:numerical} indicate that our algorithm converges asymptotically superlinearly and thus $\theta$ should be less than $\frac{1}{2}$ for our tests.
\end{remark}

\section{Numerical Results}\label{sec:numerical}

\subsection{Implementation}

Since the variational problem \eqref{eq:main} is non-convex, the starting point is significant for both the solution quality and convergence speed of our algorithms. Throughout our tests, we simply set $b^0=0$ and chose $(\mbfu^0, \mbfv^0)$ as follows.  

Let $\mbfx^{av}=\frac{1}{n}\sum_{i=1}^n\mbfx_i$. Then set $\mbfu^0$ to the negative of the first $r$ left singular vectors and $\mbfv^0$ to the first $r$ right singular vectors of $\mbfx^{av}$ corresponding to its first $r$ largest singular values. 

The intuition of choosing such $(\mbfu^0,\mbfv^0)$ is that it is one minimizer of $\frac{1}{n}\sum_{i=1}^n \tr(\mbfu^\top\mbfx_i\mbfv)$, which is exactly the first-order Taylor expansion of $\ell(\mbfu,\mbfv,0)$ at the origin, under constraints $\mbfu^\top\mbfu=\mbfi$ and $\mbfv^\top\mbfv=\mbfi$. Unless specified, the algorithms were terminated if they ran over 500 iterations or the relative error $q^k \le 10^{-3}$.

\subsection{Scalability}

In order to demonstrate the computational benefit of the proximal method, we compared Algorithm~\ref{alg:bcpd} with Algorithm~\ref{alg:bcd} on randomly generated data. Each data point\footnote{We use synthetic data simply for scalability and speed test.  For other numerical experiments, we use real-world datasets.} in class ``+1'' was generated by MATLAB command \verb|randn(s,t)+1| and each one in class ``-1'' by \verb|randn(s,t)-1|. The sample size was fixed to $n=100$, and the dimensions were kept by $s=t$ with $s$ varying among $\{50,100,250,500,750,1000\}$. We tested two sets of parameters for the scalability test. We ran each algorithm with one set of parameters for 5 times with different random data.  

Table \ref{table:scal} shows the average running time and the median number of iterations. From the table, we see that both Algorithm~\ref{alg:bcd} and Algorithm~\ref{alg:bcpd} are scalable to large-scale dataset and converge within the given tolerance after quite a few iterations. The per-iteration running time increases almost linearly with respect to the data size. In addition, Algorithm~\ref{alg:bcpd} is much faster than Algorithm~\ref{alg:bcd} in terms of running time.  Note the degree of speedup depends on the parameters. In the first experiment, where $\ell_2$ regularization dominates ($\mu_1 = \nu_1 = 0.1$, $\mu_2 = \nu_2 = 1$), Algorithm~\ref{alg:bcpd} is twice as fast as Algorithm~\ref{alg:bcd}.  In the second experiment, where $\ell_1$ regularization dominates ($\mu_1 = \nu_1 = 0.1$, $\mu_2 = \nu_2 = 0$), Algorithm~\ref{alg:bcpd} is about 20 times faster than Algorithm~\ref{alg:bcd}.

\begin{table}[t]
\caption{Scalability and comparison of Algorithms~\ref{alg:bcd} and~\ref{alg:bcpd}.  Shown are the average running time and median number of iterations.}
\vskip 1ex
\begin{center}
\resizebox{0.56\columnwidth}{!}{\begin{tabular}{|c|cc||cc|}\hline
& \multicolumn{2}{|c||}{Algorithm \ref{alg:bcd}} & \multicolumn{2}{|c|}{Algorithm \ref{alg:bcpd}}\\\hline
\multicolumn{5}{|c|}{$\mu_1 = \nu_1 = 0.1$, $\mu_2 = \nu_2 = 1$}\\\hline
$(s,t)$ & time (sec.) & iter & time (sec.) & iter\\\hline
$(50,50)$ & 0.79 & 5 & \bf 0.03 & 9 \\ 
$(100,100)$ & 1.13 & 6 & \bf 0.06 & 11 \\
$(250,250)$ & 3.89 & 6 & \bf 0.56 & 31 \\
$(500,500)$ & 9.96 & 5 & \bf 1.80 & 4 \\
$(750,750)$ & 18.60 & 7 & \bf 4.04 & 4 \\
$(1000,1000)$ & 16.25 & 3 & \bf 7.92 & 4 \\
\hline
\multicolumn{5}{|c|}{$\mu_1 = \nu_1 = 0.1$, $\mu_2 = \nu_2 = 0$}\\\hline
$(s,t)$ & time (sec.) & iter & time (sec.) & iter\\\hline
$(50,50)$ & 6.87 & 17 & \bf 0.37 & 282 \\ 
$(100,100)$ & 14.39 & 29 & \bf 0.38 & 47 \\
$(250,250)$ & 21.73 & 8 & \bf 3.49 & 28 \\
$(500,500)$ & 78.32 & 7 & \bf 4.07 & 11 \\
$(750,750)$ & 129.23 & 8 & \bf 4.31 & 4 \\
$(1000,1000)$ & 218.49 & 9 & \bf 8.19 & 4 \\ \hline
\end{tabular}}
\end{center}
\label{table:scal}
\end{table}

\subsection{Convergence Behavior}

We ran Algorithm~\ref{alg:bcpd} up to 600 iterations for the unregularized model ($\mu_1=\nu_1=\mu_2=\nu_2=0$), and $10^4$ iterations for the regularized model where we set $\mu_1=\nu_1=0.01$ and $\mu_2=\nu_2=0.5$. For both models, $r=1$ was used. The last iterate was used as $\mbfw^*$. The dataset is described in Section 6.1.1.  

Figure~\ref{fig:convg} shows the convergence behavior of Algorithm~\ref{alg:bcpd} for solving \eqref{eq:main} with different regularization terms. From the figure, we see that our algorithm converges pretty fast and the difference $\|\mbfw^k-\mbfw^*\|_F$ appears to decrease linearly at first and superlinearly eventually.

\begin{figure}[t]
\begin{center}
\includegraphics[width=0.58\columnwidth]{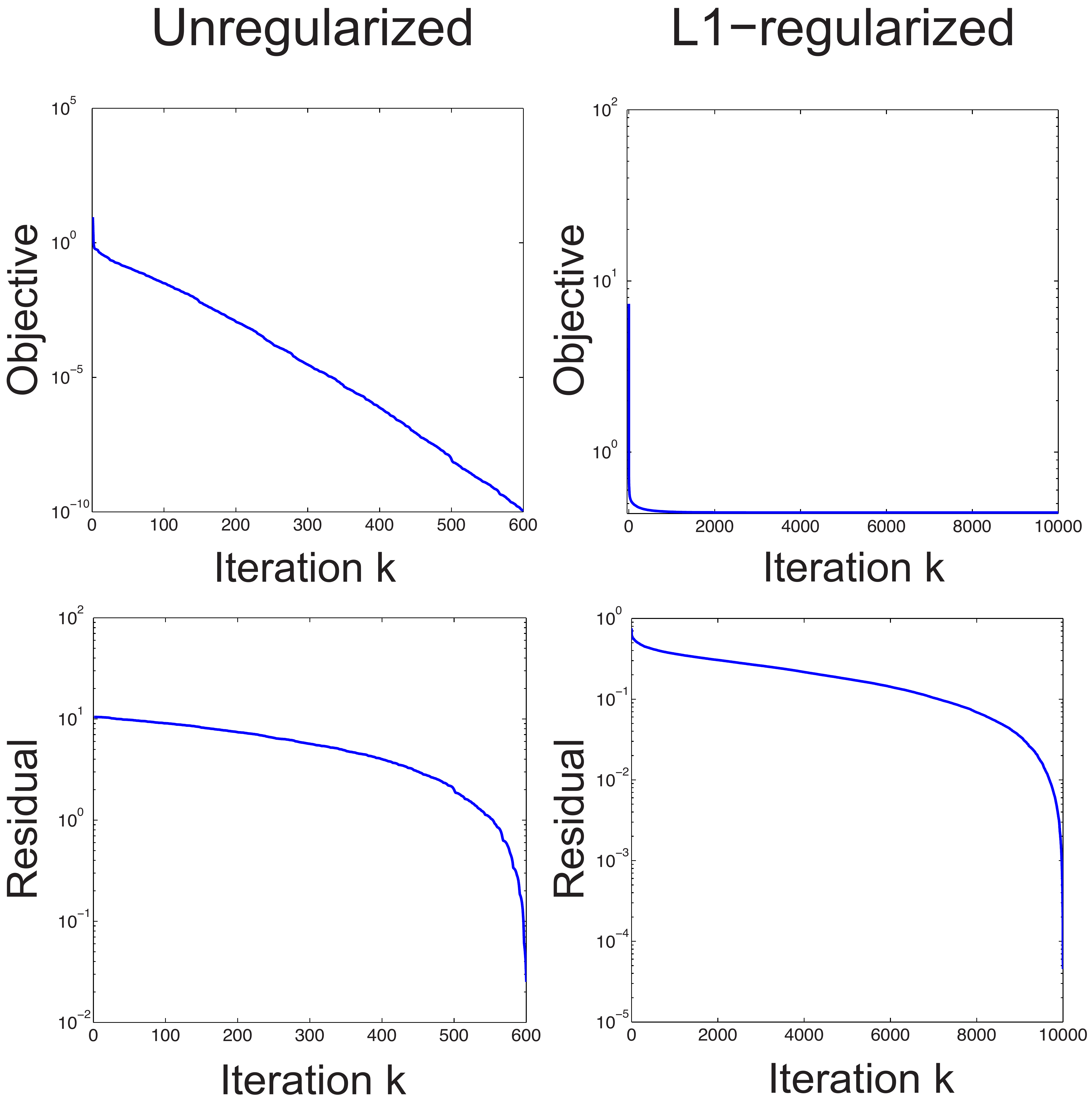}
\end{center}
\vskip -2ex
\caption{Convergence behavior for solving \eqref{eq:main} using Algorithm~\ref{alg:bcpd}. Top panel plots the objective function as a function of iteration.  Bottom panel plots the residual $\|\mbfw^k-\mbfw^*\|_F$ as a function of iteration.}
\label{fig:convg}
\end{figure}
\vskip -2ex

\section{Applications}

We apply sparse bilinear logistic regression to several real-world applications and compare its generalization performance with logistic regression, sparse logistic regression and bilinear logistic regression. We also extend the sparse bilinear logistic regression from the binary case to multi-class case in several experiments.

\subsection{Brain Computer Interface}

\subsubsection{Binary Case}

We tested the classification performance of sparse bilinear logistic regression \eqref{eq:main} on an EEG dataset with binary labels.  We used the EEG dataset IVb from from BCI competition III \footnote{\url{http://www.bbci.de/competition/iii/}} . Dataset IVb concerns a motor imagery classification task.  The 118 channel EEG was recorded from a healthy subject sitting in a comfortable chair with arms resting on armrests. Visual cues (letter presentation) were shown for 3.5 seconds, during which the subject performed: left hand, right foot, or tongue. The data was sampled at 100 Hz, and the cues of ``left hand'' and ``right foot'' were marked in the training data. We chose all the $210$ marked data points for test and downsampled each point to have 100 temporal slices, namely, $s=118$, $t=100$ in this test.

In \eqref{eq:main}, there are five parameters $\mu_1,\mu_2,\nu_1,\nu_2$ and $r$ to be tuned. Leave-one-out cross validation was performed on the training dataset to tune these data. First, we fixed $\mu_1=\mu_2=\nu_1=\nu_2=0$ (i.e., unregularized) and tuned $r$. Then, we fixed $r$ to the previously tuned one ($r=1$ in this test) and selected the best $(\mu_1,\mu_2,\nu_1,\nu_2)$ from a $6\times 5\times 6\times 5$ grid. 

\begin{table}[t]
\caption{Classification performance for the BCI EEG dataset.} 
\vskip 1ex
\begin{center}
\resizebox{0.65\columnwidth}{!}{\begin{tabular}{|c|c|}\hline
Models & Prediction Accuracy \\\hline
Logistic Regression & 0.75 \\
Sparse Logistic Regression & 0.76 \\
Bilinear Logistic Regression & 0.84 \\
Sparse Bilinear Logistic Regression & \textbf{0.89} \\
\hline
\end{tabular}}
\end{center}
\label{table:eeg-predict}
\end{table}

Table~\ref{table:eeg-predict} shows the prediction accuracy on the testing dataset.  We used the ROC analysis to compute the Az value (area under ROC curve) for both the unregularized model and the regularized model, where the best hyperparameters for the regularized model are tuned on the validation dataset using cross validation.  We compared (sparse) logistic regression with (sparse) bilinear logistic regression.  We solved the $\ell_1$-regularized logistic regression using FISTA~\cite{beck2009}.  We observed that bilinear logistic regression gives much better predictions than logistic regression. In addition, sparse bilinear logistic regression performs better than the unregularized bilinear logistic regression.

\subsubsection{Multi-class Case}

\begin{table}[h]
\caption{Classification performance for the multi-class EEG dataset.}
\begin{center}
\vskip 1ex
\resizebox{0.65\columnwidth}{!}{\begin{tabular}{|c|c|}\hline
Models & Prediction Accuracy \\\hline
Logistic Regression & 0.54 \\
Sparse Logistic Regression & 0.54 \\
Bilinear Logistic Regression & 0.55 \\
Sparse Bilinear Logistic Regression & \textbf{0.65} \\
\hline
\end{tabular}}
\end{center}
\label{table:multieeg-predict}
\end{table}

We further extended our sparse bilinear logistic regression to the multi-class case using one-versus-all method.  The EEG dataset in this experiment was based on a cognitive experiment where the subject view images of three categories and tried to make a decision about the category~\cite{lou2011}.  The data was recorded at 2048 Hz using a 64-channel EEG cap.  We downsampled this data to 100 Hz.  

Table~\ref{table:multieeg-predict} shows classification performance for the multi-class classification.  Consistently for all the three stimuli, bilinear logistic regression outperforms logistic regression, and sparse bilinear logistic regression further improves the generalization performance by introducing sparsity.

\subsection{Separating Style and Content}

As mentioned earlier, one benefit of the bilinear model is to separate style and content.  In order to exploit this property, we classified images with various camera viewpoints.  

\begin{figure}[h]
\begin{center}
\includegraphics[width=0.58\columnwidth]{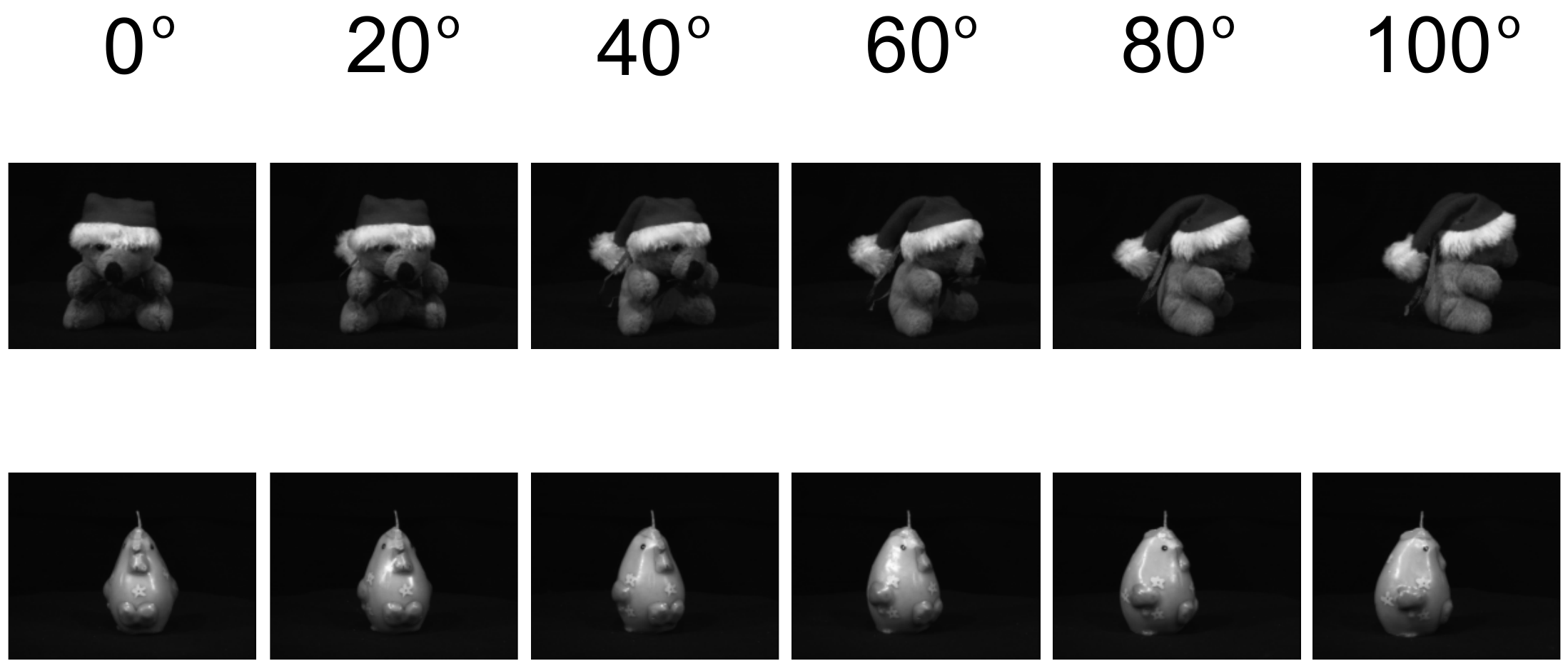}
\end{center}
\vskip -2ex
\caption{Sample images with various camera viewpoints.}
\label{fig:imgview}
\end{figure}

We used the Amsterdam Library of Object Images,\footnote{\url{http://staff.science.uva.nl/~aloi/}} where the frontal camera was used to record 72 viewpoints of the objects by rotating the object in the plane at $5^{\circ}$ resolution from $0^{\circ}$ to $355^{\circ}$.  Figure~\ref{fig:imgview} shows some sample images with various camera viewpoints.

\begin{table}[h]
\caption{Classification performance for images with various camera viewpoints.} 
\begin{center}
\vskip 1ex
\resizebox{0.65\columnwidth}{!}{\begin{tabular}{|c|c|}\hline
Models & Prediction Accuracy \\\hline
Logistic Regression & 0.86 \\
Sparse Logistic Regression & 0.86 \\
Bilinear Logistic Regression & 0.94 \\
Sparse Bilinear Logistic Regression & \textbf{1.00} \\
\hline
\end{tabular}}
\end{center}
\label{table:imgview}
\end{table}

Table~\ref{table:imgview} shows the comparison between (sparse) logistic regression and (sparse) bilinear logistic regression.  We observe a significant improvement using the bilinear model, and sparse bilinear logistic regression achieves the best generalization performance.

\subsection{Visual Recognition of Videos}

We used sparse bilinear logistic regression to videos~\cite{pirsiavash2009}, in the context of visual recognition for UCF sports action dataset.\footnote{\url{http://crcv.ucf.edu/data/UCF\_Sports\_Action.php}}  Since the size of the original video is big, we reduced the dimensionality of feature space by extracting histograms based on scale-invariant feature transform (SIFT) descriptors~\cite{lowe1999} for each frame.  

\begin{figure}[h]
\begin{center}
\includegraphics[width=0.58\columnwidth]{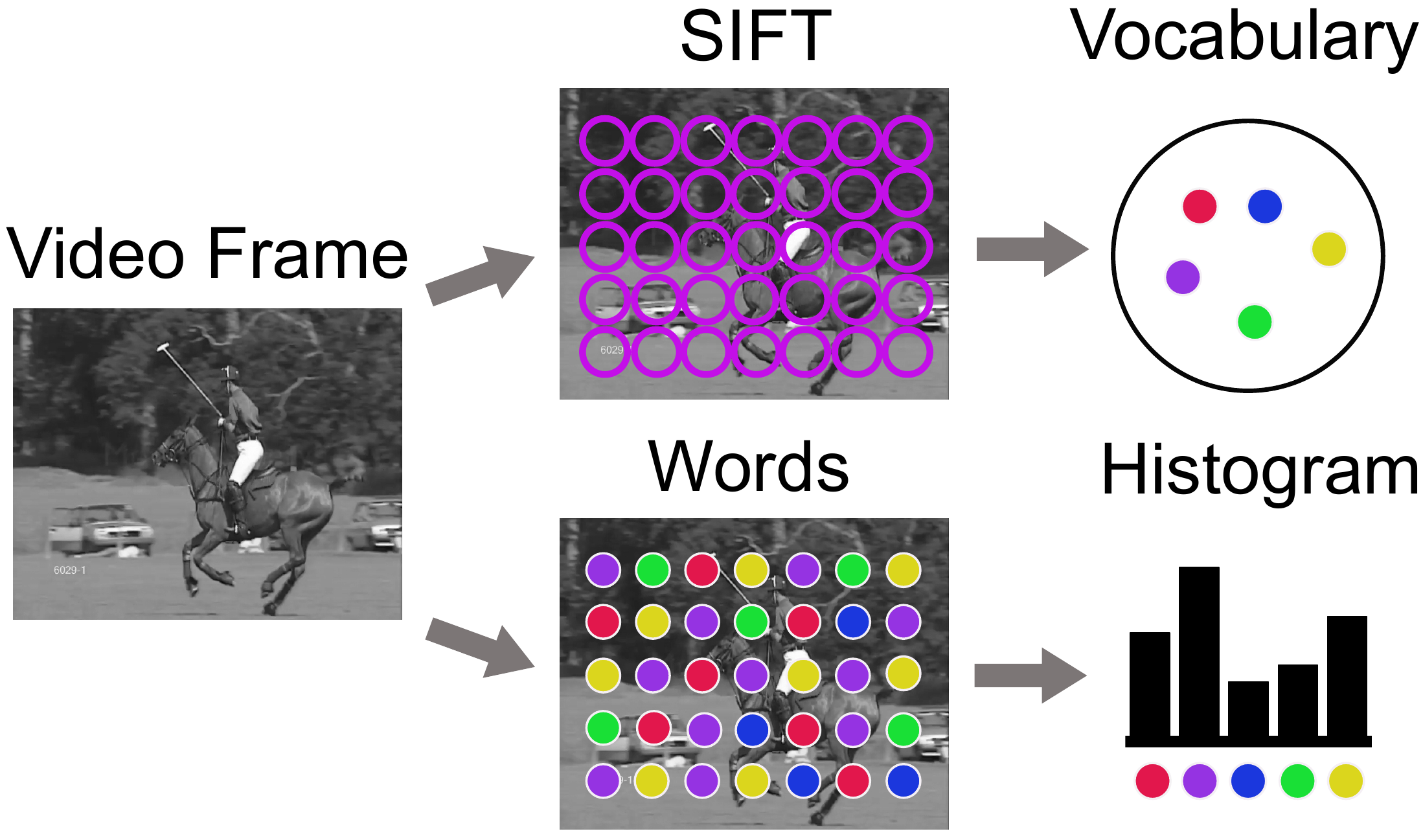}
\end{center}
\vskip -2ex
\caption{Illustration of building SIFT histogram features.}
\label{fig:sift}
\end{figure}

Figure~\ref{fig:sift} illustrates such a procedure.  We first built a vocabulary for the codebook assuming $100$ words, using k-mean clustering based on all the SIFT descriptors across frames for all the videos.  We then constructed histograms for each frame according to the codebook.  A tiling technique was used to improve the performance.  This procedure reduced the feature space to $s = 400$ and $t = 55$. 

We focused on five classes of sports action and we used the following abbreviations: Diving (Diving-Side), Riding (Riding-Horse), Run (Run-Side), Swing (Swing-Sideangle), Walk (Walk-Front).  We picked 6 videos out of each class, and used 6-fold cross validation to test discrimination accuracy in the context of transfer learning.  

\begin{table}[h]
\caption{Classification performance for the UCF sports action video dataset.}
\begin{center}
\vskip 1ex
\resizebox{0.65\columnwidth}{!}{\begin{tabular}{|c|c|}\hline
Models & Prediction Accuracy \\\hline
Logistic Regression & 0.70 \\
Sparse Logistic Regression & 0.70 \\
Bilinear Logistic Regression & 0.73 \\
Sparse Bilinear Logistic Regression & \textbf{0.77} \\
\hline
\end{tabular}}
\end{center}
\label{table:conf_mat_ucf}
\end{table}

Table~\ref{table:conf_mat_ucf} shows the classification performance for (sparse) logistic regression and (sparse) bilinear logistic regression. In overall, sparse bilinear logistic regression achieves the best classification performance.

\section{Discussion}

We proposed sparse bilinear logistic regression, and developed an efficient numerical algorithm using the block coordinate proximal descent method.  Theoretical analysis revealed its global convergence as well as convergence rate.  We demonstrated its generalization performance on several real-world applications.

\subsection{Dimensionality Reduction and Classification}

It should be noted that bilinear logistic regression performs dimensionality reduction and classification within the same framework.  Traditionally in order to combat the curse of dimensionality, dimension reduction techniques such as principle component analysis (PCA) and independent component analysis (ICA) were commonly used as a preprocessing step before carrying out classification.  Instead of a two-step processing, bilinear logistic regression carries out dimension reduction and classification using one optimization problem.  

Sparse bilinear logistic regression further fuses the benefits of sparse logistic regression and bilinear logistic regression into the same framework.  Sparsity overcomes the ambiguity intrinsic to the bilinear model, which is critical to the quality of solution.  Sparsity leads to feature selection in both spatial and temporal domains.  More importantly, sparsity improves the generalization performance of the classifier, which is intimately related to the logarithmic sample complexity~\cite{ng2004}.  We demonstrated such an improvement using a range of numerical experiments.  However, it remains a challenging problem to carry out rigorous statistical analysis based on the minimax theory due to the bi-convex nature.

\subsection{Bi-convexity: More Gain than Pain}

Bilinear model introduces bi-convexity into the objective function, however, it should be noted that the resulting decision boundary is still linear.  Recall the objective function for sparse bilinear logistic regression is the following
\begin{equation*}
\min_{\mbfu,\mbfv,b} \frac{1}{n}\sum_{i=1}^n\log \left( 1+\exp[-y_i (\tr(\mbfu^\top\mbfx_i\mbfv)+b)] \right)+r_1(\mbfu)+r_2(\mbfv).
\end{equation*}
The estimated spatial factor $\mbfu \in \mathbb{R}^{s \times r}$ and temporal factor $\mbfv \in \mathbb{R}^{r \times t}$ essentially forms a low-rank weight matrix $\mbfw \in \mathbb{R}^{s \times t}$, $\mbfw = \mbfu \mbfv^{\top}$.  Hence the decision boundary can be written as
\begin{equation*}
\langle \mathrm{diag}(\mbfw \otimes \mathbf{1}), x \rangle + b = 0.
\end{equation*}

With such an interpretation, one can also reformulate the objective function of the sparse bilinear logistic regression as
\begin{equation*}
\min_{\mbfw,\mbfu,\mbfv,b} \frac{1}{n}\sum_{i=1}^n\log \left( 1+\exp[-y_i (\tr( \mbfw \otimes \mbfx_i)+b)] \right)+\lambda\|\mbfw\|_*+r_1(\mbfu)+r_2(\mbfv), \quad \mbfw = \mbfu \mbfv^{\top}.
\end{equation*}
We see that the bilinear logistic regression has an equivalent convex formulation by minimizing the nuclear norm of $\mbfw$.  However, due to the benefits of sparsity (as discussed in section 2.3.2), it becomes critical to have a bilinear factorization and impose sparsity promoting priors on the spatial and temporal factors. 

As much as the difficulty of statistical analysis posed by bi-convexity, our numerical algorithm for solving sparse bilinear logistic regression is extremely efficient and has a guarantee for global convergence.  As we demonstrated empirically on a variety of classification tasks, sparse bilinear logistic regression provides an avenue to boost generalization performance.

\subsection{Multinomial Generalization}

The binomial sparse bilinear logistic regression can be further generalized to the multinomial case.  We assume each sample $\{\bfx_i\}$ to belong to $(m+1)$ classes and label $y_i\in\{1,2,\cdots,m+1\}$ and seek $(m+1)$ hyperplanes $\{\bfx: \bfw_c^\top\bfx+b_c=0\}_{c=1}^{m+1}$ to separate these samples.  According to the logistic model, the conditional probability for $y_i$ based on sample $\bfx_i$ is
\begin{equation}\label{eq:multi-pr}
P(y_i=c|\bfx_i,\bfw,\bfb)=\frac{\exp[\bfw_c^\top\bfx_i+b_c]}{\sum_{j=1}^{m+1}\exp[\bfw_j^\top\bfx_i+b_j]},\quad c = 1,\cdots,m+1.
\end{equation}
Because of the normalization condition $\sum_{c=1}^{m+1}P(y_i=c|\bfx_i,\bfw,\bfb)=1$, one $(\bfw_c,b_c)$ needs not be estimated. Without loss of generality, we set $(\bfw_{m+1},b_{m+1})$ to \emph{zero}. 
Let $y_{ic}=1$ if $y_i=c$ and $y_{ic}=0$ otherwise. Then \eqref{eq:multi-pr} becomes
\begin{equation}
P(y_i|\bfx_i,\bfw,\bfb)=\frac{\exp[\sum_{c=1}^my_{ic}(\bfw_c^\top\bfx_i+b_c)]}{1+\sum_{c=1}^{m}\exp[\bfw_c^\top\bfx_i+b_c]}.
\end{equation}

The average negative log-likelihood function is 
\begin{eqnarray*}
\label{eq:multi-ml}
\mcl(\bfw,\bfb) &=& -\frac{1}{n}\sum_{i=1}^n\log P(y_i|\bfx_i,\bfw,\bfb)\nonumber\\
&=& \frac{1}{n}\sum_{i=1}^n\left(\log\big(1+\sum_{c=1}^{m}\exp[\bfw_c^\top\bfx_i+b_c]\big)
-\sum_{c=1}^my_{ic}(\bfw_c^\top\bfx_i+b_c)\right)
\end{eqnarray*}
To perform MLE for $(\bfw,\bfb)$, one can minimize $\mcl(\bfw,\bfb)$. Under the above setting, where each sample is a matrix and each weight $\bfw_c$ has the form of $\mbfu_c\mbfv_c^\top$, the loss function becomes
\begin{equation*}\label{eq:multi-bml}
\mcl(\bmcu,\bmcv,\bfb)=\frac{1}{n}\sum_{i=1}^n\left(\log\big(1+\sum_{c=1}^{m}\exp[\tr(\mbfu_c^\top\mbfx_i\mbfv_c)+b_c]\big)
-\sum_{c=1}^my_{ic}(\tr(\mbfu_c^\top\mbfx_i\mbfv_c)+b_c)\right).
\end{equation*}
The multinomial sparse bilinear logistic regression takes the following variational formulation
\begin{equation}\label{eq:multi-reg}
\min_{\bmcu,\bmcv,\bfb}\mcl(\bmcu,\bmcv,\bfb)+
R_1(\bmcu)+R_2(\bmcv),
\end{equation}
where $\bmcu=(\mbfu_1,\cdots,\mbfu_m), \bmcv=(\mbfv_1,\cdots,\mbfv_m)$ with $\mbfu_c\in\mbr^{S\times K}$ and $\mbfv_c\in\mbr^{T\times K}$ for each class $c$, and $R_1$ and $R_2$ are used to promote priori structures on $\bmcu$ and $\bmcv$, respectively.

\bibliographystyle{plain}
\bibliography{BLRref}

\end{document}